\documentclass[a4paper,reqno, 11pt]{amsart}  
\usepackage[DIV=12, oneside]{typearea}
\usepackage[utf8]{inputenc}
\usepackage[T1]{fontenc}
\usepackage[english]{babel}
\usepackage[centertags]{amsmath}
\usepackage{amstext,amssymb,amsopn,amsthm}
\usepackage{nicefrac, esint}
\usepackage{mathrsfs}
\usepackage{dsfont}
\usepackage{bbm}
\usepackage{thmtools}
\usepackage{graphicx}

\usepackage[backgroundcolor=white, bordercolor=blue,
linecolor=blue]{todonotes}
\usepackage[colorlinks=true, linkcolor=blue, citecolor=black]{hyperref}
\usetikzlibrary{patterns}
\usetikzlibrary{arrows}
\usepackage{enumitem}

\addto\extrasenglish{}
\addto\extrasenglish{}

\declaretheorem[name=Theorem, numberwithin=section]{theorem}
\newtheorem{corollary}[theorem]{Corollary}
\newtheorem{lemma}[theorem]{Lemma}
\newtheorem{proposition}[theorem]{Proposition}
\newtheorem{definition}[theorem]{Definition}

\newtheorem{assumptions}{Assumption}

\newtheorem*{example*}{Example}
\newtheorem*{robremark*}{Robustness remark}

\declaretheoremstyle[bodyfont=\normalfont]{remark-style}

\numberwithin{equation}{section}

\theoremstyle{plain}

\newcommand{\N}{\mathds{N}}
\newcommand{\R}{\mathds{R}}
\newcommand{\C}{\mathds{C}}

\newcommand{\diag}{\mathrm{diag}}

\newcommand{\BIGOP}[1]
{
\mathop{\mathchoice%
{\raise-0.22em\hbox{\huge $#1$}}%
{\raise-0.05em\hbox{\Large $#1$}}{\hbox{\large $#1$}}{#1}}}

\def\XXint#1#2#3{{\setbox0=\hbox{$#1{#2#3}{\int}$}
     \vcenter{\hbox{$#2#3$}}\kern-.5\wd0}}

\newcommand{\BIGboxplus}{\mathop{\mathchoice%
{\raise-0.35em\hbox{\huge $\boxplus$}}%
{\raise-0.15em\hbox{\Large $\boxplus$}}{\hbox{\large $\boxplus$}}{\boxplus}}}

\DeclareMathOperator{\supp}{supp}

\renewcommand{\d}{\textnormal{d}}

\begin{document}
\allowdisplaybreaks
\title{The martingale problem for a class of nonlocal operators of diagonal type}

\author{Jamil Chaker}

\address{Fakult\"{a}t f\"{u}r Mathematik\\Universit\"{a}t Bielefeld\\Postfach 
100131\\D-33501 Bielefeld}



\email{jchaker@math.uni-bielefeld.de}

\begin{abstract}
We consider systems of stochastic differential equations of the form
\[ \d X_t^i = \sum_{j=1}^d A_{ij}(X_{t-}) \d Z_t^j\]
for $i=1,\dots,d$ with continuous, bounded and non-degenerate coefficients. Here $Z_t^1,\dots,Z_t^d$ are independent one-dimensional stable processes 
with $\alpha_1,\dots,\alpha_d\in(0,2)$. 
In this article we research on uniqueness of weak solutions to such systems by studying the corresponding martingale problem.  
We prove the uniqueness of weak solutions in the case of diagonal coefficient matrices.
\end{abstract}
\maketitle                   
\section{Introduction}
Anisotropies and discontinuities are phenomena of great interest that arise in several natural and financial models. 
In this paper we bring together these two subjects and investigate in anisotropic nonlocal operators.
We research on diagonal systems of stochastic differential equations driven by pure jump L\'evy processes with anisotropic L\'evy measures. 

Let $d\in\N$, $d\geq 3$. For $j\in\{1,\dots,d\}$, let $(Z_t^j)_{t\geq 0}$ be an one-dimensional symmetric stable L\'evy process of order $\alpha_j \in (0,2)$, i.e.
$Z_t^j$ is a pure jump L\'evy process with L\'evy-Khintchine triplet $(0,0,c_{\alpha_i}|h|^{-1-\alpha_j} \d h)$, where the constant $c_{\alpha_j}$ is chosen such that
\[ \mathds{E}[e^{i\xi Z_t^j}]=e^{-t|\xi|^{\alpha_j}} \quad \text{for } t>0 \text{ and } \xi\in\R.  \]
We assume $Z_t^j$'s to be independent and define the $d$-dimensional process $Z_t=(Z_t^1,\dots,Z_t^d)$.
Consider the following system
 \begin{equation}\label{SDEM}
 \begin{aligned}
      &\displaystyle dX^i_{t} = \sum_{j=1}^d A_{ij}(X_{t-})dZ_{t}^j, \quad \text{ for } i\in\{1,\dots,d\}, \\
   &X_0 = x_0,   
   \end{aligned} 
   \end{equation}
where $A:\R^d\to\R^{d\times d}$ is a matrix-valued function, which is pointwise non-degenerate and has bounded continuous entries. 
We mainly study this system in the case of diagonal matrices, i.e. $A_{i,j}\equiv 0$ whenever $i\neq j$.
The aim of this paper to prove uniqueness of solutions to the martingale problem for diagonal systems.
If the matrix-valued function $A$ is global Lipschitz continuous, it is easy to prove that there is a unique strong solution by Picard's iteration. However, when
$A$ is only continuous, the uniqueness of weak solutions becomes a challenging problem. 
Since L\'evy processes are appropriate processes for modeling price processes for instance, such stochastic differential equations are interesting objects in financial mathematics.

The case where $\alpha_1=\cdots=\alpha_d=\alpha\in(0,2)$ has been studied systematically in the non-diagonal case in \cite{BCS} for non-diagonal coefficient matrices. Furthermore, if $Z_t$ consists of independent $2$-stable L\'evy processes, the process $Z_t$ is a $d$-dimensional Brownian motion running twice the speed, which is well studied. 

We study the diagonal system with the method of Stroock and Varadhan from 1969, see \cite{STROOCKVARA}.
The celebrated martingale problem provides an equivalent concept of 
existence and uniqueness in law for weak solutions to stochastic differential equations.
An overview of the martingale problem for elliptic operators in non-divergence form can be found in \cite[Chapter 6]{STROOCK} or \cite[Chapter VI]{BAD}.

Up to the present day, the martingale problem is still an intensely studied topic. 
For instance, in \cite{ABELSKAS} unique solvability of the Cauchy problem for a class 
of integro-differential operators is shown to imply the well-posedness
of the martingale problem for the corresponding operator. In \cite{CHENZHANG} the authors study well-posedness of the martingale problem for a class of stable-like operators and in
\cite{PRIOLA} the author considers degenerate stochastic differential equations and proves weak uniqueness of solutions using the martingale problem.
In \cite{KUEHN}, existence and uniqueness for stochastic differential equations driven by L\'evy processes and stable-like processes with unbounded coefficients are studied.
For an overview of the martingale problem for jump processes, see \cite[Chapter 4]{Jacob} and the references therein.

Let us give a short survey to known results related to our studies.\\
In \cite{BASSUNIQ}, Bass considers nonlocal operators of the form
\[ Lf(x)=\int_{\R\setminus\{0\}} (f(x+h)-f(x)-f'(x)h\mathds{1}_{[-1,1]}(h)) \, \nu(x,\d h) \]
for $f\in C_b^2(\R)$ and gives sufficient conditions on $\nu$ for the existence and uniqueness of a solution to the martingale problem for $L$. 
Further, the author proves that the associated stochastic process $X_t$ is a Feller process with respect to the unique solution $\mathds{P}^x$ to the martingale problem for $L$ started at $x\in\R$.
One important example in the paper is $\nu(x,\d h)\asymp |h|^{-1-\alpha(x)}\d h$ with $0<\inf\{\alpha(x)\colon x\in\R\} \leq \sup\{\alpha(x)\colon x\in\R\}<2$. In this case the Dini 
continuity of $\alpha:\R\to(0,2)$ is a sufficient condition for well-posedness of the martingale problem.
Although the results in the paper were proven in one spatial dimension they can be extended to higher dimensions. 
In \cite{Schilling}, the authors present sufficient conditions for the transience and the existence of local times for Feller processes. The studies contain the class of stable-like processes of
\cite{BASSUNIQ}. With a different method the authors prove a transience criterion and the existence of local times for these kind of processes in $d$ dimensions for $d\in\N$.\\
In \cite{HOH1}, Hoh considers operators of a similar form, but the starting point is a different representation of the operator.
The author studies the operator as a pseudo-differential operator of the form
\[ Lf(x)=-p(x,D)f(x)=-(2\pi)^{-d/2}\int_{\R^d}e^{ix\cdot \xi}p(x,\xi)\cdot \widehat{f}(\xi)\, \d\xi \]
for $f\in C_0^{\infty}(\R^d)$, where for any fixed $x\in\R^d$, $p(x,\cdot)$ is negative definite. 
In the paper, uniqueness of the martingale problem for pseudo-differential operators with the symbol $p(x,\xi)$ of the form
\[ p(x,\xi)=-\sum_{i=1}^d b_i(x)a_i(\xi) \]
is studied, where $b_i$, $i\in\{1,\dots,d\}$ are non-negative, bounded and $d+m$ times continuously differentiable for some $m\in\N$ and $a_i$, $i\in\{1,\dots,d\}$ are
continuous non-negative definite with $a_i(0)=0$. This covers for example symbols of the type
\[ -\sum_{i=1}^d b_i(x)|\xi_j|^{\alpha_j} \]
for $\alpha_j\in(0,2]$ with the already mentioned conditions on $b_i$. \\
In \cite{BCS} the authors study systems of stochastic differential equations of the form \eqref{SDEM},
where the driving process $Z_t=(Z_t^1,\dots,Z_t^d)$ consists of $d$ independent copies 
of a one-dimensional symmetric stable process of index $\alpha\in(0,2)$. Hence the indices of stability are the same in every direction, which is the main difference to our model.
In the article the authors prove existence and well-posedness of the martingale problem for
\[ \mathcal{L}f(x)=\sum_{j=1}^d \int_{\R\setminus\{0\}} \left(f(x+a_j(x)w)-f(x)-w\mathds{1}_{|w|\leq 1}\nabla f(x)\cdot a_j(x)\right)\frac{c_{\alpha}}{|w|^{1+\alpha}} \, \d h, \]
where $f\in C_b^2(\R^d)$ and $a_j(x)$ denotes the j$^{\text{th}}$ column of $A(x)$ and show that this is equivalent to existence and uniqueness of solutions to \eqref{SDEM}.
This operator is a pseudo-differential operator of the form
\[ \mathcal{L}f(x)=-\int_{\R^d}p(x,\xi)e^{-ix\cdot \xi} \widehat{f}(\xi)\, \d\xi, \]
where
\begin{equation}\label{symbolp}
p(x,\xi)=-\sum_{i=1}^d |\xi\cdot a_j(x)|^{\alpha}.
\end{equation} 
In \cite{SchilSchn}, Schilling and Schnurr study stochastic differential equations of the form
$dX_t=\Phi(X_{t-})dL_t, X_0=x_0\in\R^d$, where $\Phi:\R^d\to\R^{d\times n}$ is Lipschitz continuous with linear growth and $L_t$ is a $\R^n$-valued L\'evy process. They prove that the unique strong solution of this equation has the symbol given by 
$p(x,\xi)=\Psi(\Phi(x)^t\xi)$, where $\Psi$ is the symbol of $L_t$ (see \cite[Theorem 3.1]{SchilSchn}).
Note that the symbol \eqref{symbolp} is exactly of the form $p(x,\xi)=\Psi(A(x)^t\xi)$ as in \cite{SchilSchn}, where $A(x)=(a_{ij}(x))_{i,j=1,\dots,d}$ and $\Psi(\xi)=-\sum_{i=1}^d |\xi_i|^{\alpha}$.

Because of the lack of differentiability, the symbol \eqref{symbolp} does not fit into the set-up of \cite{HOH1}. The authors' central idea is the usage of a perturbation argument as in \cite{STROOCK}. 
The main part of the paper is the proof of an $L^p$-boundedness result of pseudo-differential operators whose symbols have the form
\[-\frac{|a(x)\cdot \xi|^{\alpha}}{\sum_{i=1}^d |\xi_i|^{\alpha}}, \]
where $a:\R^d\to\R^d$ is continuous and bounded with respect to the Euclidean norm from above and below by positive constants. The proof of this $L^p$-boundedness follows from 
a result by Calder\'on and Zygmund, see \cite{CALDERON}. 
Therefore, the main difficulty is to show that the operator fits into the set-up of \cite{CALDERON}, which is done with the method of rotations.\\
In \cite{BCR}, Bass and Chen study the same system of stochastic differential equations and prove
H\"older regularity of harmonic functions with respect to $\mathcal{L}$. 
Furthermore, they give a counter example and thus show that the Harnack inequality for harmonic functions is not fulfilled.
In \cite{Kulc} the authors study \eqref{SDEM} in the case of diagonal matrices $A$ and $\alpha_1=\alpha_2=\dots=\alpha_d=\alpha\in(0,2)$. They prove 
sharp two-sided estimates of the corresponding transition density $p^{A}(t,x,y)$ and prove H\"older and gradient estimates for the function $x\mapsto p^{A}(t,x,y)$.

Our consideration of the system \eqref{SDEM} driven by the anisotropic L\'evy process consisting of independent one-dimensional stable L\'evy processes of orders that may differ leads to the operator 
\begin{equation}\label{introoper}
\mathcal{L}f(x)=\sum_{j=1}^d \int_{\R\setminus\{0\}} \left(f(x+a_j(x)w)-f(x)-w\mathds{1}_{\{|w|\leq 1\}}\nabla f(x)\cdot a_j(x)\right)\frac{c_{\alpha_j}}{|w|^{1+\alpha_j}} \, \d h. 
\end{equation}
Since the existence of weak solutions is well-studied, we shortly comment on known results from which one can derive the existence of weak solutions to \eqref{SDEM}.
If $A_{ij}$ is bounded and continuous for every $i,j\in\{1,\dots,d\}$, the existence of weak solutions to \eqref{SDEM} can be deduced for instance from \cite[Theorem 3.15]{HOHHAB} or \cite[Theorem 4.1]{KUEHN}. 
Furthermore, a direct proof of the existence of weak solutions to \eqref{SDEM} for the case $\alpha_1=\alpha_2=\dots=\alpha_d\in(0,2)$ can be found in \cite[Section 4]{BCS}. This proof can be adjusted to the case of different indices of stability.\\
In order to prove uniqueness of solutions to the martingale problem, we have to restrict ourselves to matrices whose entries are zero outside the diagonal.  
The reason we have to restrict ourselves to this restrictive case, is that the approach of Calder\'on and Zygmund seems to be not applicable. We have not been able to prove a general non-diagonal version of the required $L^p$-boundedness result (c.f.  Proposition \ref{BR0}) and leave it as a conjecture. Once this result is proven the techniques we use would lead to uniqueness of solutions to the martingale problem.\\
In the case of diagonal matrices, we have to prove an $L^p$-boundedness result for 
pseudo-differential operators $\mathcal{B}$ of the form
\[ \mathcal{B}f(x) = \int_{\R^d}\left(\sum_{k=1}^d\frac{c_{\alpha_k}|A_{kk}(x)\xi_k|^{\alpha_k}}{\sum_{i=1}^d |C_i\xi_i|^{\alpha_i}}\right)e^{-ix\cdot \xi} \widehat{f}(\xi)\, \d\xi \]
for some constants $C_i\neq 0$ in order to apply the perturbation argument as in \cite{STROOCK}.

The main ingredient in the proof of the $L^p$-bound for the perturbation operator $\mathcal{B}$ is a Fourier multiplier theorem 
which goes back to Ba\~nuelos and Bogdan, see \cite[Theorem 1]{BANFOUR}. 
In order to apply this result we have to show that the perturbation operator $\mathcal{B}$ is an operator on $L^2(\R^d)$ with the representation
\[ \widehat{\mathcal{B}f}(\xi) =  \frac{\int_{\R^d} (\cos(\xi\cdot z)-1) \phi(z) \, V(\d z)}{\int_{\R^d} (\cos(\xi\cdot z)-1) \, V(\d z)} \widehat{f}(\xi) \]
for a measurable and bounded function $\phi:\R^d\to\C$ and a positive L\'evy measure $V$. This allows us to prove the well-posedness of the martingale problem for $\mathcal{L}$.
\begin{theorem}\label{thm:uniqueness}
 Suppose $A$ satisfies $A_{ij}\equiv0$ for $i\neq j$, $x\mapsto A_{jj}(x)$ is bounded continuous for all $j\in\{1,\dots,d\}$ and $A(x)$ is non-degenerate for any $x\in\R^d$.
 For every $x_0\in\R^d$, there is a unique solution to the martingale problem for $\mathcal{L}$ started at $x_0\in\R^d$.
\end{theorem}

The anisotropic system \eqref{SDEM} has been studied in \cite{JAMIL}, where harmonic functions are shown to satisfy an H\"older estimate.

\subsection*{Notation} 
Let $A\subset\R^d$ be open. We denote by
$C(A)$ the space of all continuous functions on $A$, by
$C_b(A)$ the space of all continuous and bounded functions on $A$, by $C_c(A)$ the space of all continuous functions on $A$ with compact support and by $C_0(A)$ the space of all continuous functions on $A$ vanishing at infinity.
Furthermore, let 
$C_b^2(A)$ be the space of bounded continuous functions on $A$ that have continuous bounded derivatives up to second order. 
Similarly let $C_0^2(A)$ and $C_c^2(A)$ denote the space of all functions in $C_0(A)$ resp. $C_c(A)$ with derivatives up to second order in $C_0(A)$ resp. $C_c(A)$.
The space of all smooth and compactly supported functions on $A$ is denoted by $C_c^{\infty}(A)$.
 For $i\in\N$ we write $c_i$ for positive constants and additionally $c_i=c_i(\cdot)$ if we want to highlight all the quantities the constant depends on.
\subsection*{Structure of the article} This work is organized as follows. In Section \ref{secpre} we provide definitions, constitute sufficient preperation and 
illustrate the underlying L\'evy process.
Section \ref{chaptersec} consists of four subsections. In Subsection \ref{fixed} we study the system of stochastic differential equations with fixed coefficients and prove in 
Subsection \ref{bddsub} boundedness of the resolvent operators for weak solutions to the system. Subsection \ref{auxsub} contains auxiliary results and in Subsection \ref{proofsub} we present the proof of the uniqueness of weak solutions to the system. 
\section{Preliminaries}\label{secpre}
The aim of this section is to give a brief exposition of the system of stochastic differential equations \eqref{SDEM}. 

Let us start with some basic definitions and notations. \\
We denote the Skorohod space of all c\`adl\`ag functions on $[0,\infty)$ with values in $\R^d$ 
by $\mathds{D}([0,\infty))$. For $f\in\mathds{D}([0,\infty))$ let 
\[ f(t-):=\lim\limits_{s\nearrow t} f(s) \quad \text{and} \quad \Delta f(t) = f(t)-f(t-). \]
Let $(\breve{Z_t^i})_{t\geq 0}$ be the $d$-dimensional L\'{e}vy process, defined by $\breve{Z_t^i} = Z_t^ie_i$, where $e_i$ is the $i^{\text{th}}$ standard coordinate vector.
Then $(\breve{Z_t^i})_{t\geq 0}$ is a L\'{e}vy process with L\'{e}vy?Khintchine triplet $(0,0,\breve{\nu_i}(\d h))$, where $\breve{\nu_i}$ is given by
\[ \breve{\nu_i}(\d w) = \frac{c_{\alpha_i}}{|w_i|^{1+\alpha_i}}\, \d w_i \left(\prod_{j\neq i}\delta_{\{0\}}(\d w_j)\right). \]
Obviously, $(Z_t)_{t\geq 0}$ is the sum of the $d$ independent L\'{e}vy processes $\breve{Z_t^i}$, $i=1,\dots,d$ and hence a L\'{e}vy process itself. \\
Using the independence of the $\breve{Z_t^i}$'s, the L\'{e}vy-measure of $(Z_t)_t$ is given as the sum of the $\breve{\nu_i}$'s, i.e.
\[  \nu(\d w)= \sum_{i=1}^{d} \left( \frac{c_{\alpha_i}}{|w_i|^{1+\alpha_i}}\,\d w_i \left(\prod_{j\neq i}\delta_{\{0\}}(\d w_j) \right)\right).\]
The support of this measure is the union of the coordinate axes. Hence $\nu(A)=0$ for every set $A\subset \R^d$, which has an empty intersection with the coordinate axes.\\
The process $Z_t$ makes a jump into the i$^{\text{th}}$ direction of the coordinate axis, whenever $Z_t^i$ makes a jump.\\
For $f\in C_b^2(\R^d)$ let
\begin{equation}\label{gen}
 \mathcal{L}f(x)=\sum_{j=1}^{d} \int_{\R\setminus \{0\}} (f(x+a_{j}(x)h)-f(x)-h\mathds{1}_{\{|h|\leq 1\}}\nabla f(x) \cdot a_j(x))\frac{c_{\alpha_j}}{|h|^{1+\alpha_j}} \d h, 
\end{equation}
where $a_j(x)$ denotes the j$^{th}$ column of the matrix $A(x)$.\\
Let us recall the concept of weak solutions and solutions to the martingale problem.
\begin{definition}
We call a filtered probability space $(\Omega,\mathcal{F},(\mathcal{F}_t)_{t\geq 0}, \mathds{P})$ and stochastic processes 
$(X_t^1,\dots, X_t^d)$ and $(Z_t^1,\dots,Z_t^d)$ weak solution to the system \eqref{SDEM}, starting at $x_0$, if \eqref{SDEM} holds and the processes 
$(Z_t^i)_{i=1,\dots,d}$ are independent one-dimensional symmetric stable processes of index $\alpha_i$ under $\mathds{P}$.
\end{definition}
In particular, the existence of a unique weak solution to \eqref{SDEM} implies that for a given initial distribution, the law of $(X_t)_{t\geq 0}$ is uniquely determined.
For the sake of brevity, we denote weak solutions by $(\Omega,\mathcal{F},(\mathcal{F}_t)_{t\geq 0}, \mathds{P}, X, Z)$.\\
An equivalent formulation of the concept of weak solutions to stochastic differential equations is given by the so-called martingale problem method
which gives us another approach to study solvability of stochastic differential equations by studying the corresponding generator.
\begin{definition}\label{defmart}
Let $\mathcal{L}$ be an operator whose domain includes $C_b^2(\R^d)$.
Let $(X_t)_{t\geq 0}$ be the coordinate maps on $\Omega=\mathds{D}([0,\infty))$, that is $X_{t}(\omega)=\omega(t)$ and $(\mathcal{F}_t)_{t\geq 0}$ be the filtration generated by $(X_t)_{t\geq 0}$. We say a probability measure $\mathds{P}$ is a solution to the martingale problem for $\mathcal{L}$, started at $x_0$, 
if the following two conditions hold:
\begin{enumerate}
 \item $\mathds{P}(X_0=x_0)=1.$
 \item For each $f\in C_b^2(\R^d)$
 \[  f(X_t)-f(X_0)-\int_{0}^{t} \mathcal{L}f(X_s) \d s \] 
 is a $\mathds{P}$-martingale.
 \end{enumerate} 
\end{definition}
We make the following assumptions on the coefficients to the system \eqref{SDEM}, c.f. \cite[Assumption 2.1.]{BCS}.
\begin{assumptions}\label{assumption1}{\ }
\begin{enumerate}
 \item  For every $x\in\R^d$ the matrix $A(x)$ is non-degenerate, that is 
 \[ \inf\limits_{u\in\R^d \colon |u|=1} |A(x)u|>0. \]
  \item The functions $x\mapsto A_{ij}(x)$ are continuous and bounded for all $1\leq i,j\leq d$.
\end{enumerate}
\end{assumptions}

The system \eqref{SDEM} has been studied in the case $\alpha_1=\alpha_2=\cdots=\alpha_d=\alpha\in(0,2)$ by Bass and Chen in their articles \cite{BCS}, \cite{BCR}. 
In \cite{BCS} the authors prove with the help of the martingale problem the existence and uniqueness of a weak solution to \eqref{SDEM} in the 
case $\alpha_1=\alpha_2=\cdots=\alpha_d=\alpha\in(0,2)$. The main tool to obtain uniqueness is by using a bound on the $L^p$-operator norm of the corresponding perturbation integral operator. 
In order to do so the authors use the method of rotations, which seems not to be applicable in the case of different indices.\\
In \cite{BCR} the authors study bounded harmonic functions for the corresponding integral operator and show that such harmonic functions are H\"older continuous. This result has been 
extended in \cite{JAMIL} for the case where the $\alpha_i$'s are allowed to be different.\\
Let us first write the integro-differential operator $\mathcal{L}$ with weighted second order differences.
\begin{lemma}\label{second-order}
 Let $f\in C_b^2(\R^d)$, then
 \[ \mathcal{L}f(x) = \frac{1}{2}\sum_{j=1}^d \int_{\R\setminus\{0\}} (f(x+a_j(x)h)-2f(x)+f(x-a_j(x)h))\frac{c_{{\alpha_j}}}{|h|^{1+\alpha_j}}\,\d h.\]
\end{lemma}
\begin{proof}
Follows immediately by symmetry.
\end{proof}

Let $(\Omega,\mathcal{F},(\mathcal{F}_t)_{t\geq 0}, \mathds{P}, X, Z)$ be a weak solution to \eqref{SDEM}. Proposition \ref{martingale} shows that any weak solution to \eqref{SDEM} is a solution to the martingale problem for the operator $\mathcal{L}$.
Since there are no significant differences in the proof of Proposition \ref{martingale} and the proof of \cite[Proposition 4.1.]{BCS}, we skip the proof and refer the reader to \cite{BCS}.
\begin{proposition} \label{martingale}
 Suppose $A$ is bounded and measurable. Let $(\Omega,\mathcal{F},(\mathcal{F}_t)_{t\geq 0}, \mathds{P}, X, Z)$ be a weak solution to \eqref{SDEM}.
 If $f\in C_b^2(\R^d)$, then
 \[ f(X_t)-f(X_0)-\int_0^t \mathcal{L}f(X_s) \,\d s \]
 is a $\mathds{P}$-martingale.
\end{proposition}
\section{Uniqueness}\label{chaptersec}
In this section we prove uniqueness of weak solutions to \eqref{SDEM}. For this purpose we add the
following assumption on the coefficients.
\begin{assumptions}\label{diagonal}
 Suppose $A_{ij}(x)=0$ for all $x\in\R^d$, whenever $i\neq j$.
\end{assumptions}
The aim of this section is to prove Theorem \ref{thm:uniqueness}.
For this purpose we first study the system \eqref{SDEM}, where the coefficients are fixed, and use a perturbation argument in the spirit of \cite{STROOCK}.
\subsection{Perturbation}\label{fixed}
Let $x_0=(x_0^1,\dots,x_0^d)\in\R^d$ be a fixed point.  
We define the process $(U_t)_{t\geq 0}$ by
\[ U_t=U_0+A(x_0)Z_t. \]
Note that $(U_t)_{t\geq 0}$ is an affine transformation of a L\'{e}vy process and has stationary and independent increments and c\`adl\`ag paths. 
Hence $(U_t-U_0)_{t\geq 0}$ is a L\'{e}vy process. 
For $f\in C_b^2(\R^d)$ consider the operator
\begin{equation}\label{generatorL_0}
  \mathcal{L}_0f(x)=\sum_{j=1}^{d} \int_{\R\setminus \{0\}} (f(x+e_jA_{jj}(x_0)h)-f(x)-h\mathds{1}_{\{|h|\leq 1\}}\partial_jf(x)A_{jj}(x_0))\frac{c_{{\alpha_j}}}{|h|^{1+\alpha_j}} \, \d h.
\end{equation}
For a function $f\in L^1(\R^d)$, we define its Fourier transform by
\[ \mathcal{F}f(\xi):=\widehat{f}(\xi):=\int_{\R^d} e^{ix\cdot \xi} f(x)\, \d x, \quad \xi\in\R^d. \]
For $j\in\{1,\dots,d\}$, let
\begin{equation}
  \mathcal{I}_jf(x):=\int_{\R\setminus \{0\}} (f(x+e_jA_{jj}(x_0)h)-f(x)-h\mathds{1}_{\{|h|\leq 1\}}\partial_jf(x)A_{jj}(x_0))\frac{c_{{\alpha_j}}}{|h|^{1+\alpha_j}} \, \d h.
\end{equation}
\begin{lemma}\label{multiplier1}
Let $f\in C_c^{\infty}(\R^d)$. Then
 \[ \widehat{\mathcal{L}_0f}(\xi) = -\sum_{j=1}^d|\xi_j A_{jj}(x_0)|^{\alpha_j}\widehat{f}(\xi) .\]
\end{lemma}
\begin{proof}
By the substitution $w=(\xi_j A_{jj}(x_0))h$ and Fubini's theorem we get
\begin{align*}
\widehat{\mathcal{I}_jf}(\xi) &= \frac12\int_{\R^d}e^{i\xi\cdot x} \int_{\R\setminus \{0\}} (f(x+e_jA_{jj}(x_0)h)-2f(x)+f(x-e_jA_{jj}(x_0)h))\frac{c_{{\alpha_j}}}{|h|^{1+\alpha_j}} \, \d h\, \d x \\
& = \frac12\widehat{f}(\xi)\int_{\R\setminus \{0\}}(e^{ih\xi_jA_{jj}(x_0)}-2+e^{-ih\xi_jA_{jj}(x_0)})\frac{c_{{\alpha_j}}}{|h|^{1+\alpha_j}} \, \d h \\
& = \widehat{f}(\xi)\int_{\R\setminus \{0\}}(\cos(h\xi_jA_{jj}(x_0))-1) \frac{c_{{\alpha_j}}}{|h|^{1+\alpha_j}}\, \d h \\
& = -\widehat{f}(\xi)\int_{\R\setminus \{0\}}(1-\cos(h\xi_jA_{jj}(x_0))) \frac{c_{{\alpha_j}}}{|h|^{1+\alpha_j}}\, \d h \\
& = -\widehat{f}(\xi)|(\xi_jA_{jj}(x_0))|^{\alpha_j}\int_{\R\setminus \{0\}}(1-\cos(w)) \frac{c_{{\alpha_j}}}{|w|^{1+\alpha_j}}\, \d w.
\end{align*}
The choice of $c_{{\alpha_j}}$ yields
\[ \int_{\R\setminus \{0\}}(1-\cos(w)) \frac{c_{{\alpha_j}}}{|w|^{1+\alpha_j}}\, dw=1, \]
which proves the assertion.
\end{proof}
Note that by Lemma \ref{multiplier1} the characteristic function of $(U_t^j)_{t\geq 0}$ is given by
 \begin{align*}
 \mathds{E}\left(e^{i\xi\cdot U_t^j}\right) & := \exp(-t\Psi_j(\xi_j)) =\exp\left( -t|\xi_jA_{jj}(x_0)|^{\alpha_j} \right).
\end{align*}
Since the $Z_t^i$'s are independent, the $U_t^i$'s are also independent.
Therefore the characteristic function of $(U_t)_{t\geq 0}$ is the product of the characteristic functions of $(U_t^j)_{t\geq 0},$ i.e. 
 \begin{equation}\label{characteristicfunction}
 \mathds{E}\left(e^{i\xi\cdot U_t}\right) = \exp\left(-t\sum_{j=1}^d\Psi_j(\xi_j)\right) =: \exp(-t\Psi(\xi)) .
\end{equation}

Next we show a scaling result for the transition density function of $(U_t)_{t\geq 0}$. 
It is reasonable to first study the transition density of $Z_t$, 
since $(U_t)_{t\geq 0}$ is given as an affine transformation of $(Z_t)_{t\geq 0}$. \\
By Lemma \ref{multiplier1}, we deduce that the characteristic function of $(Z_t^j)_{t\geq 0}$ is given by
 \begin{align*}
 \mathds{E}\left(e^{i\xi\cdot Z_t^j}\right) & := \exp(-t\psi_j(\xi_j)) =\exp\left( -t|\xi_j|^{\alpha_j} \right).
\end{align*}
Note that, since $\exp(-\psi_j(\cdot))\in L^1(\R)$, the inverse Fourier-Transform exists and therefore the 
transition density function $q_t^j$ of $Z_t^j$ exists and is given by
\begin{equation}\label{transdens}
 q_t^j(x_j)=\mathcal{F}^{-1}\left( e^{-t\psi_j} \right)(x_j)=\frac{1}{(2\pi)}\left(\int_{\R}e^{-ix_jy_j}e^{-t\psi_j(y_j)} \, \d y\right).
\end{equation}
Hence
\begin{equation}\label{transdens1}
 \int_{\R} e^{ix\xi_j} q_t^j(x)\, \d x = \exp(-t|\xi_j|^{\alpha_j}).
\end{equation}
Since the processes $Z_t^j$, $j\in\{1,\dots,d\}$ are independent and $Z_t=(Z_t^1,\dots,Z_t^d)$, 
the transition density function of $(Z_t)_{t\geq 0}$ is given by
\begin{equation}\label{transitionprod}
 q_t(x) = \prod_{j=1}^d q_t^j(x_j).
\end{equation}
The scaling property for one-dimensional symmetric $\alpha_j$-stable processes states\\
 $q_t^j(x_j) = t^{-1/\alpha_j} q_1(t^{-1/\alpha_j}x_j)$. See e.g. \cite[Chapter 8]{BERT} 
for more details.  \\
Using this scaling property of the one-dimensional processes and \eqref{transitionprod} we get the following scaling property for the transition density of $(Z_t)_{t\geq 0}$
\[ q_t(x) = \prod_{j=1}^d q_t^j(x_j) = \prod_{j=1}^d t^{-1/\alpha_j} q_1^j(t^{-1/\alpha_j}x_j) = t^{-\sum_{k=1}^d 1/\alpha_k} q_1(t^{-1/\alpha_1}x_1,\dots,t^{-1/\alpha_d}x_d).\]
We next deduce a scaling property for the transition density of $(U_t)_{t\geq 0}$.\\
Let $B\in\mathcal{B}(\R^d)$. Then by using the substitution $z=A(x_0)x+U_0$
\begin{align*}
 \mathds{P}(U_t\in B) & =  \mathds{P}(U_0 + A(x_0)Z_t \in B) = \mathds{P}(A(x_0)Z_t \in B-U_0) = \mathds{P}(Z_t \in A(x_0)^{-1}(B-U_0)) \\
 & = \int_{\R^d} q_t(x) \, \mathds{1}_{A(x_0)^{-1}(B-U_0)}(x) \, \, dx \\
 & = \frac{1}{|\det(A(x_0))|}\int_{B} q_t(A(x_0)^{-1}(z-U_0)) \, dz.
\end{align*}
Set 
\begin{equation}\label{transitiontilde}
  p^j_t(x) = \frac{1}{A_{jj}(x_0)} q_t^j((A(x_0)^{-1}(x-U_0))_j).
\end{equation}
Then the transition density $p_t(x)$ of $(U_t)_{t\geq 0}$ is given by
\begin{equation}\label{transitionY}
\begin{aligned}
  \frac{1}{|\det(A(x_0))|}q_t(A(x_0)^{-1}(x-U_0)) &= \frac{1}{|\det(A(x_0))|} \prod_{j=1}^d q_t^j((A(x_0)^{-1}(x-U_0))_j) \\
  = \prod_{j=1}^d p^j_t(x) = p_t(x).
\end{aligned}
\end{equation}
Moreover, we have 
\begin{equation}\label{scalingpt}
  p_t(x) = \frac{t^{-\sum_{k=1}^d 1/\alpha_k}}{|\det(A(x_0))|}q_1(\varXi(t)(A(x_0)^{-1}(x-U_0))),
\end{equation}
where 
\[ \varXi(t) = \diag(t^{-1/\alpha_1},\dots,t^{-1/\alpha_d})= \begin{pmatrix}
t^{-1/\alpha_1}	& 0	& \dots	 & 0      \\
0	& t^{-1/\alpha_2} 	& \dots  & 0 	  \\
\vdots	& 0 	& \ddots & \vdots \\
0 	& \dots & 0	 & t^{-1/\alpha_d}
\end{pmatrix}. \]
We define the transition semigroup $(P_t)_{t\geq 0}$ of $(U_t)_{t\geq 0}$ and $(Q_t)_{t\geq 0}$ of $(Z_t)_{t\geq 0}$ on $C_{0}(\R^d)$ by
\[ P_tf(x)= \int_{\R^d}p_t(x-y)f(y) \, \d y=\mathds{E}[f(U_t+x)]=:\mathds{E}^{x}[f(U_t)] \]
and
\[ Q_tf(x)= \int_{\R^d}q_t(x-y)f(y) \, \d y=\mathds{E}[f(Z_t+x)]=:\mathds{E}^{x}[f(Z_t)]. \]
By e.g. \cite[Theorem 31.5]{SAT} the operators $\{P_t : t\geq 0\}$ indeed define a strongly continuous semigroup on $C_0(\R^d)$ with operator norm $\|P_t\|=1$.
Note that $P_tf$ is also well-defined for $f\in C_b(\R^d)$ but, in general, $(P_t)_{t\geq 0}$ is not strongly continuous on $C_b(\R^d)$.  
We now state an important result on the limit behavior of the semigroup.
\begin{theorem}\label{convhg}
 Let $f\in C_0(\R^d)$. Then
 \begin{equation}
  \lim\limits_{t\to\infty} \|P_tf\|_\infty = 0.
 \end{equation}
 \end{theorem}
\begin{proof}
Let $\epsilon>0$. Choose $R>1$ such that $|f(y)| \leq \epsilon/2$ for all $y\in \R^d\setminus B_R(0)$.
 Then
 \begin{align*}
  P_tf(x)& =\int_{\R^d}p_t(x-y)f(y) \, \d y = \int_{B_R(0)}p_t(x-y)f(y) \, \d y + \int_{\R^d\setminus B_R(0)}p_t(x-y)f(y) \, \d y\\
  & :=(I)+(II).
 \end{align*}
For each $t\geq 0$, we have 
\[ (II)\leq \frac{\epsilon}{2} \int_{\R^d\setminus B_R(0)}p_t(x-y) \, \d y \leq \frac{\epsilon}{2} \underbrace{\int_{\R^d}p_t(x-y) \, \d y}_{=1} = \frac{\epsilon}{2}. \]
Moreover by \eqref{scalingpt}, we know there is a consant $c_1>0$ such that
\[ p_t(x) \leq c_1\frac{t^{-\sum_{k=1}^d 1/\alpha_k}}{|\det(A(x_0))|}. \]
Thus, 
\begin{align*}
 (I)  &= \int_{B_R(0)}p_t(x-y)f(y) \, \d y 
  \leq c_1 \frac{t^{-\sum_{k=1}^d 1/\alpha_k}}{|\det(A(x_0))|}\int_{B_R(0)}f(y) \, \d y \\
 & \leq c_1\|f\|_\infty \frac{t^{-\sum_{k=1}^d 1/\alpha_k}}{|\det(A(x_0))|} |B_R(0)|.
\end{align*}
Choose $t_0 \geq 0$ such that for all $t\geq t_0$
\[ \|f\|_\infty \frac{t^{-\sum_{k=1}^d 1/\alpha_k}}{|\det(A(x_0))|} |B_R(0)|\leq \frac{\epsilon}{2}. \]
Note that the choice of $t_0$ is independent of $x$. Hence the assertion follows.
\end{proof}
We introduce some important operators associated to the family of operators $(P_t)_{t\geq 0}$ on $C_b(\R^d)$.
From \eqref{transdens} and \eqref{transitiontilde} we immediately see $p_t(z)=p_t(-z)$ for every $z\in\R^d$. \\
Hence there exists a positive and symmetric potential density function $r_\lambda$ with respect to $(U_t)_{t\geq 0}$, that is
\begin{equation}\label{potential density}
 0<r_{\lambda}(y-x)=r_\lambda (x-y):=\int_0^\infty e^{-\lambda t}p_t(x-y) \, \d t.
\end{equation}
Let $f\in C_b(\R^d)$. For $\lambda>0$ we define the $\lambda$-resolvent operator of $(U_t)_{t\geq 0}$ by
\begin{equation}\label{resolventop}
 R_{\lambda} f(x) := \int_{\R^d} f(y)r_{\lambda}(x-y) \, \d y = \int_0^\infty e^{-\lambda t} P_tf(x) \, \d t = \mathds{E}^x \left[\int_0^\infty e^{-\lambda t} f(U_t)\, \d t\right].
\end{equation}
The resolvent operator describes the distribution of the process evaluated at independent exponential times.
That is, if $\tau=\tau(\lambda)$ has exponential law with parameter $\lambda>0$ and is independent of $(U_t)_{t\geq 0}$, then
\[ \mathds{E}[f(U_{\tau})] = \lambda R_\lambda f. \]
It is often more convenient to work with the resolvent operators than with the semigroup, thanks to the smoothing effect
of the Laplace transform and to the lack of memory of exponential laws. \\
To study these objects in detail, we first have to define for $\lambda\geq0$ the $\lambda$-potential measures $V^{\lambda}(x,\cdot)$, $x\in\R^d$ on 
$\mathcal{B}(\R^d)$ by
\begin{equation}
 V^{\lambda}(x,B) = \mathds{E}^x \left[ \int_0^\infty e^{-\lambda t}\mathds{1}_{\{U_t\in B\}} \, \d t \right] \ \text{ for } B\in\mathcal{B}(\R^d).
\end{equation}
Note that $V^\lambda$ is obviously well-defined for all $\lambda>0$.
Since $U_t(\omega)$ is measurable in $(0,\infty)\times\Omega$, the application of Fubini's theorem implies
\begin{align*}
 V^\lambda(B)&=\mathds{E}\left[ \int_0^\infty e^{-\lambda t} \mathds{1}_B(U_t) \, \d t \right] =  \int_0^\infty e^{-\lambda t} \mathds{E}\left[\mathds{1}_B(U_t)\right] \, \d t \\
 &= \int_0^\infty e^{-\lambda t} \mathds{P}(U_t\in B) \, \d t \leq \int_0^\infty e^{-\lambda t} \, \d t=\frac{1}{\lambda}.
\end{align*}
Clearly, this argument is not valid for $\lambda=0$. This case will be studied separately at a later point. 
The $0$-potential measure will be denoted by $V(x,B)$ and is called potential measure.
By \eqref{resolventop} and the definition of $V^{\lambda}$ we obtain an additional representation of the $\lambda$-resolvent operator on $C_b(\R^d)$ by
\[ R_{\lambda}f(x) = \int_{\R^d} f(y) V^{\lambda}(x,\, \d y). \]
Recall the following properties.
\begin{lemma}\label{resolventident}
Let $f\in C_b(\R^d)$. Then we have the following properties.
 \begin{enumerate}
  \item $R_{\lambda}f-R_{\mu}f = (\mu-\lambda)R_{\lambda}R_{\mu}f$ \, for $\lambda,\mu>0$,
  \item $R_{\lambda}R_{\mu}f=R_{\mu}R_{\lambda}f$.
 \end{enumerate}
\end{lemma}
Let us define 
\begin{equation}\label{beta}
 \beta=\sum_{j=1}^d \frac{1}{\alpha_j}.
\end{equation}
Let us prove some elementary facts about $R_{\lambda}$ and $P_t$.
\begin{proposition}\label{resolventenabsch} Let $\lambda>0$ and $f\in C_b(\R^d)$.
\begin{enumerate}
 \item  If $p\in[1,\infty]$, then
 \[ \|R_{\lambda}f\|_p \leq \frac{\|f\|_p}{\lambda}. \]
 \item  If $p\in(1,\infty]$, then
 \begin{equation}\label{Ptestimate}
   |P_tf(x)| \leq t^{-\beta/p}\|p_1(\cdot)\|_q\|f\|_p, 
 \end{equation}
 where $q$ is the conjugate exponent to $p$.
 \item  If $p>\beta$, then
 \[ |R_{\lambda}f(x)| \leq c_1\|f\|_p, \]
 where $c_1=\|p_1(\cdot)\|_q \int_0^\infty e^{-\lambda t} t^{-\beta/p} dt.$
\end{enumerate}
\end{proposition}
\begin{proof}
The idea of the proof goes back to \cite[Proposition 2.2]{BCS}.\\
Without loss of generality we can assume $f\in L^p(\R^d)$. Otherwise the assertions are trivially true.
\begin{enumerate}
 \item   By Young's inequality and the conservativeness of $p_t$, we have
 \[ \|P_tf\|_p \leq \|p_t\|_1\|f\|_p=\|f\|_p. \]
 Therefore by Minkowski's inequality
 \[ \|R_\lambda f\|_p \leq \int_0^\infty e^{-\lambda t} \|P_tf\|_p\, \d t \leq \frac{1}{\lambda} \|f\|_p. \]
 \item  By H\"older's inequality,
 \[ |P_tf(x)|=\left|\int_{\R^d} p_t(x-y) f(y)\, \d y \right| \leq \|f\|_p \|p_t(x-\cdot)\|_q=\|f\|_p\|p_t(\cdot)\|_q. \]
 Using the scaling property for $p_t$, we get
 \[ p_t(x) = \frac{t^{-\beta}}{|\det(A(x_0))|}q_1(\varXi(t)(A(x_0)^{-1}(x-U_0))). \]
 Hence 
 \begin{align*}
  \|p_t(\cdot)\|_q & = \| \frac{t^{-\beta}}{|\det(A(x_0))|}q_1(\varXi(t)(A(x_0)^{-1}(\cdot-U_0)))\|_q \\
  & =  \frac{t^{-\beta}}{|\det(A(x_0))|}\left(\det(\varXi(t))^{-1}\right)^{1/q} \|q_1((A(x_0)^{-1}(\cdot-U_0)))\|_q \\
  & =  \frac{t^{-\beta}t^{\beta/q}}{|\det(A(x_0))|}\|q_1((A(x_0)^{-1}(\cdot-U_0)))\|_q \\
  & =  t^{-\beta\frac{q-1}{q}}\|p_1(\cdot)\|_q = t^{-\beta/p}\|p_1(\cdot)\|_q.
 \end{align*}
 \item Using the previous estimate,
\begin{align*}
 |R_\lambda f(x)| = \left| \int_{\R^d} e^{-\lambda t} P_tf(x) \, \d t \right| \leq \|p_1(\cdot)\|_q \left(\int_0^\infty e^{-\lambda t} t^{-\beta/p} \, \d t\right)\|f\|_p.
\end{align*}
\end{enumerate}
\end{proof}
Let $\mathcal{A}$ be the infinitesimal generator of the semigroup $P_t$ on $C_0(\R^d)$ with domain $D(\mathcal{A})$.
Note that $\mathcal{L}_0=\mathcal{A}$ on $C_0^2(\R^d)$.\\
We now study $\lambda$-potential measures for the case $\lambda=0$. First, we give the definition of the potential operator.
\begin{definition}
 The potential operator $(N,D(N))$ for $(P_t)_{t\geq 0}$ is the operator on $C_0(\R^d)$, defined by
 \[ Nf(x)= \lim\limits_{t\to\infty} \int_0^t P_sf(x)\, \, \d s, \]
 where $f\in D(N):=\{f\in C_0(\R^d)\colon Nf$ exists in $ C_0(\R^d)\}$. 
\end{definition}
Next we state a proposition, which shows that $(N,D(N))$ plays the role of an "inverse" operator to $-\mathcal{L}_0$. 
Let $R(N)$ denote the range of the operator $N$.
\begin{proposition}[{\cite[Proposition 11.9.]{BERG}}]\label{berg119}
The following three conditions are equivalent:
\begin{enumerate}
 \item[(i)] $D(N)$ is dense in $C_0(\R^d)$,
 \item[(ii)] $R(N)$ is dense in $C_0(\R^d)$,
 \item[(iii)] $\lim\limits_{t\to\infty} P_tf=0$ for all $f\in C_0(\R^d)$.
\end{enumerate}
When conditions (i)-(iii) are fulfilled the potential operator is a densely defined, closed operator in $C_0(\R^d)$, and the infinitesimal generator $A$ 
is injective and satisfies
\[ N=-\mathcal{A}^{-1} \quad \text{and} \quad \mathcal{A} = -N^{-1}. \]
\end{proposition}
An important object will be the $0$-resolvent operator, that is the limit
\[ R_0f:=\lim\limits_{\lambda\to 0}R_\lambda f, \]
where $f\in D(R_0):=\{f\in C_0(\R^d) \colon R_0f$ exists in $C_0(R^d)\}.$ 
By \cite[Proposition 11.15]{BERG} $R_0=N$ if $\lim\limits_{t\to\infty}P_tf=0$ for all $f\in C_0(\R^d)$,
which is fulfilled by Theorem \ref{convhg}. Moreover, by Proposition \ref{berg119} $R_0$ is well-defined and a densely defined and closed operator in $C_0(\R^d)$.
\begin{lemma}\label{resolnull}
 Let $f\in D(R_0)$ and $\lambda>0$. Then
 \[ R_{\lambda}f-R_0f = -\lambda R_{\lambda}R_0f. \]
\end{lemma}
\begin{proof}
 Let $f\in D(R_0)$. By Lemma \ref{resolventident} for $\lambda,\mu>0$ we have 
 \begin{equation} \label{limresol}
 R_{\lambda}f-R_{\mu}f = (\mu-\lambda)R_{\lambda}R_{\mu}f.
 \end{equation}
 Since $f\in D(R_0)$ the limit $\lim\limits_{\mu\to 0}R_{\mu} f$ exists in $C_0(\R^d)$. Thus the result follows by taking the limit $\mu\to0$ in \eqref{limresol}. 
\end{proof}
Next we want to study the long-time behavior of the process $(U_t)_{t\geq 0}$ in terms of the potential measure, c.f. \cite{BERT}.
\begin{definition}\label{transdef}
We say that a L\'{e}vy process is transient if the potential measures are Radon measures, that is, for every compact set $K\subset\R^d$
\[ V(x,K)<\infty, \quad x\in\R^d. \]
\end{definition}
For $z\in\C$, we write $\mathcal{R}(z)$ for the real part of $z$. One method to verify transience of a L\'{e}vy process is the following. Recall
\begin{theorem}[{\cite[Theorem 17]{BERT}}\label{equivtrans}]
 Let $(L_t)_{t\geq 0}$ be a L\'{e}vy process with characteristic exponent $\Psi$. If for some $r>0$
 \begin{equation}\label{transeq}
 \int_{B_r} \mathcal{R}\left(\frac{1}{\Psi(\xi)}\right)\, \d \xi<\infty, 
 \end{equation}
then $(L_t)_{t\geq 0}$ is transient.
\end{theorem}
Note that Definition \ref{transdef} and Theorem \ref{equivtrans} also apply for shifted L\'{e}vy processes, i.e. L\'{e}vy processes whose initial value is not zero. \\
We next show that $(U_t)_{t\geq 0}$ is transient by verifying \eqref{transeq}.
\begin{proposition}\label{transient}
 $(U_t)_{t\geq 0}$ is transient.
\end{proposition}
\begin{proof}
By Theorem \ref{equivtrans} and \eqref{characteristicfunction}, if
\[ \exists r>0: \ \int_{B_r} \frac{1}{\sum_{j=1}^d|A_{jj}(x_0)\xi_j|^{\alpha_j}}\, \d \xi<\infty,\]
then $(U_t)_{t\geq 0}$ is transient. Let $r<1$ such that for $\xi\in B_{r}$. Then $|\xi_j|<1$ for any $j\in\{1,\dots,d\}$.  
Let $c_1:=\min\{ |A_{jj}(x_0)|\colon j\in\{1,\dots,d\}\}$ and $\alpha_{\max}=\max\{\alpha_j\colon j\in\{1,\dots,d\}\}$. Then
\begin{align*}
 \sum_{j=1}^d|A_{jj}(x_0)\xi_j|^{\alpha_j} &\geq c_1 \sum_{j=1}^d |\xi_j|^{\alpha_{\max}}\geq c_1 \max\limits_{j\in\{1,\dots,d\}}\{|\xi_j|^{\alpha_{\max}}\} \\
 & = c_1 \left(\max\limits_{j\in\{1,\dots,d\}}\{|\xi_j|^{2}\}\right)^{\alpha_{\max}/2} \geq \frac{c_1}{d} \left(\sum_{j=1}^d|\xi_j|^{2}\right)^{\alpha_{\max}/2} = c_2|\xi|^{\alpha_{\max}}
\end{align*}
Hence
\begin{align*}
 \int_{B_r} \frac{1}{\psi(\xi)} \, \d \xi &= \int_{B_r} \frac{1}{\sum_{j=1}^d |A_{jj}(x_0)\xi_j|^{\alpha_j}} \, \d \xi \leq c_3\int_{B_\epsilon} \frac{1}{|\xi|^{\alpha_{\max}}} \, \d \xi \\
 & = c_4\int_0^r s^{d-1}s^{-\alpha} \, \d s = c_4\int_0^r s^{d-1-\alpha} \, \d s <\infty,
\end{align*}
since $d\geq 3$ and $\alpha_{\max}\in(0,2)$ and therefore $d-\alpha_{\max}>0$.
\end{proof}
Because of the transience of the L\'evy process $U_t$ we have $R_0f(x)\to 0$ as $|x|\to\infty$ for $f\in C_c(\R^d)$, see \cite[Exercise 39.14]{SAT}.
Furthermore is easy to see that $R_0f$ for $f\in C_c(\R^d)$ is continuous by dominated convergence theorem. Hence for $f\in C_c(\R^d)$ we know
$R_0f\in C_0(\R^d)$.
We have
\begin{equation}\label{R0}
R_0f(x)=\mathds{E}^x \int_0^{\infty} f(U_s) \, \d s \quad f\in C_c(\R^d). 
\end{equation}
Different to $R_{\lambda}$ the operator $R_0$ is not well-defined on $C_b(\R^d)$. For instance let $f$ be a non-zero constant function. 
Then by the representation \eqref{R0} of $R_0$, one can easily see that $R_0f$ is infinite. \\
In the next step, we use Lemma \ref{second-order} to write the operator $\mathcal{L}_0$ on $f\in C_b^2(\R^d)$ with respect to a density.
\begin{lemma}
 Let $f\in C_b^2(\R^d)$, then
  \[ \mathcal{L}_0f(x) = \frac{1}{2}\sum_{j=1}^d \int_{\R\setminus\{0\}} (f(x+e_jh)-2f(x)+f(x-e_jh))\frac{c_{{\alpha_j}}}{|h|^{1+\alpha_j}}|A_{jj}(x_0)|^{\alpha_j} \d h.\]
\end{lemma}
\begin{proof}
 By Lemma \ref{second-order}
 \[ \mathcal{L}_0f(x) = \frac{1}{2}\sum_{j=1}^d \int_{\R\setminus\{0\}} (f(x+e_jA_{jj}(x_0)h)-2f(x)+f(x-e_jA_{jj}(x_0)h))\frac{c_{{\alpha_j}}}{|h|^{1+\alpha_j}} \, \d h. \]
 Using the substitution $t=A_{jj}(x_0)h$ for each summand, we get
 \begin{align*}
 &\mathcal{L}_0f(x) = \frac{1}{2}\sum_{j=1}^d \int_{\R\setminus\{0\}} (f(x+e_jA_{jj}(x_0)h)-2f(x)+f(x-e_jA_{jj}(x_0)h))\frac{c_{{\alpha_j}}}{|h|^{1+\alpha_j}} \, \d h \\
 &{\ } = \frac{1}{2}\sum_{j=1}^d \int_{\R\setminus\{0\}} (f(x+e_jt)-2f(x)+f(x-e_jt))\frac{c_{{\alpha_j}}}{|A_{jj}(x_0)^{-1}t|^{1+\alpha_j}}|A_{jj}(x_0)|^{-1} \, \d t \\
 &{\ } = \frac{1}{2}\sum_{j=1}^d \int_{\R\setminus\{0\}} (f(x+e_jt)-2f(x)+f(x-e_jt))\frac{c_{{\alpha_j}}}{|t|^{1+\alpha_j}}|A_{jj}(x_0)|^{\alpha_j} \, \d t,
 \end{align*}
which proves the assertion.
\end{proof}
Next we give a Fourier multiplier theorem, which goes back to \cite{BANFOUR}.\\
Given $p\in(1,\infty)$, let 
\[p^{\ast}:=\max\left\{p,\frac{p}{p-1}\right\} \quad \iff \quad p^{\ast}-1=\max\left\{(p-1),(p-1)^{-1}\right\}.\]
Let $\Pi\geq0$ be a symmetric L\'evy measure on $\R^d$ and $\phi$ a complex-valued, Borel-measurable and symmetric function with
$|\phi(z)|\leq 1$ for all $z\in\R^d$.
\begin{theorem}[{\cite[Theorem 1]{BANFOUR}}]\label{banuelosbogdan}
The Fourier multiplier with the symbol
\begin{equation}\label{multiplier}
M(\xi)=\frac{\int_{\R^d }(\cos(\xi\cdot z) -1) \phi(z)\, \Pi(\d z)}{\int_{\R^d} (\cos(\xi\cdot z) -1)\, \Pi(\d z)}
\end{equation}
is bounded in $L^p(\R^d)$ for $1<p<\infty$, with the norm at most $p^{\ast}-1$.
That is, if we define the operator $\mathcal{M}$ on $L^2(\R^d)$ by
\[ \widehat{\mathcal{M}f}(\xi) = M(\xi)\widehat{f}(\xi), \]
then $\mathcal{M}$ has a unique linear extension to $L^p(\R^d)$, $1<p<\infty$, and 
\[ \|\mathcal{M}f\|_p \leq (p^{\ast}-1)\|f\|_p. \]
\end{theorem}
For $j\in\{1,\dots,d\}$, let 
\[ \mathcal{M}_jf(x)=\int_{\R\setminus\{0\}} (f(x+e_jh)-2f(x)+f(x-e_jh))\frac{c_{{\alpha_j}}}{|h|^{1+\alpha_j}}\, \d h. \]
Our aim is to show that this operator fits into the set-up of Theorem \ref{banuelosbogdan}.
Let $f\in L^2(\R^d)$. Then 
\begin{align*}
\widehat{\mathcal{M}_jf}(\xi) & =\int_{\R^d}e^{ix\cdot \xi}\int_{\R\setminus\{0\}} (f(x+e_jh)-2f(x)+f(x-e_jh))\frac{c_{{\alpha_j}}}{|h|^{1+\alpha_j}}\, \d h\, \d x \\
& =\int_{\R\setminus\{0\}} \int_{\R^d} e^{ix\cdot \xi}(f(x+e_jh)-2f(x)+f(x-e_jh))\frac{c_{{\alpha_j}}}{|h|^{1+\alpha_j}}\, \d x\, \d h \\
& =\int_{\R\setminus\{0\}}\widehat{f}(\xi)(e^{ihe_j\cdot\xi}-2+e^{-ihe_j\cdot\xi})\frac{c_{{\alpha_j}}}{|h|^{1+\alpha_j}}\, \d h \\
& = 2\int_{\R\setminus\{0\}}\widehat{f}(\xi)(\cos(h\xi_j)-1)\frac{c_{{\alpha_j}}}{|h|^{1+\alpha_j}}\, \d h.
\end{align*} 
For $f\in C_c^2(\R^d)$, $R_0f$ is well-defined and therefore, by the previous calculation
\[ \widehat{\mathcal{M}_jR_0f}(\xi) =  -2\frac{\int_{\R\setminus\{0\}}(\cos(h\xi_j)-1)\frac{c_{{\alpha_j}}}{|h|^{1+\alpha_j}}\, \d h}{\sum_{j=1}^d\int_{\R\setminus\{0\}}(\cos(h\xi_j)-1)\frac{c_{{\alpha_j}}}{|h|^{1+\alpha_j}}|A_{jj}(x_0)|^{\alpha_j}\, \d h}\widehat{f}(\xi). \]
If we define for $z=(z_1,\dots,z_d)\in\R^d$
\begin{equation}
\begin{aligned}
\Pi(\d z)&   = \sum_{j=1}^d |A_{jj}(x_0)|^{\alpha_j} \frac{c_{\alpha_j}}{|z_j|^{1+\alpha_j}} \, \d z_j \prod_{i\neq j} \delta_{\{0\}}(\d z_i),\\
\phi(z)& = \mathds{1}_{\{z=e_ju\colon u\in\R\}}(z)|A_{jj}(x_0)|^{-\alpha_j}, 
\end{aligned}
\end{equation}
then we can write
\[ \widehat{\mathcal{M}_jR_0f}(\xi) = -2\frac{\int_{\R^d}(\cos(z\cdot\xi)-1)\phi(z)\, \Pi(\d z)}{\int_{\R^d}(\cos(z\cdot\xi)-1)\, \Pi(\d z)}\widehat{f}(\xi). \]
Therefore by Theorem \ref{banuelosbogdan} for $f\in C_c^2(\R^d)$
\begin{equation}\label{Lpj}
 \|\widehat{\mathcal{M}_jR_0f}\|_p \leq 2a(p^{\ast}-1)\|f\|_p, 
\end{equation} 
where 
\begin{equation}\label{A}
 a=\max\{|A_{11}(x_0)|^{-\alpha_1},\dots,|A_{dd}(x_0)|^{-\alpha_d}\}.
\end{equation}
Let us define the perturbation operator on $C_b^2(\R^d)$ by
\[ \mathcal{B}f(x) = \mathcal{L}f(x) - \mathcal{L}_0f(x). \]
Set
\begin{equation}\label{coeffeta}
 \eta:=\sup\limits_{j\in\{1,\dots,d\}} \||A_{jj}(\cdot)|^{\alpha_j}-|A_{jj}(x_0)|^{\alpha_j}\|_{L^{\infty}(\R^d)}.
\end{equation}
We assume
\begin{equation}\label{etaassumption}
  \eta\leq \eta_0:= \frac{1}{4da(p^{\ast}-1)}, \tag{Loc}
\end{equation}
where $a$ is defined as in \eqref{A}.
\begin{proposition}\label{BR0}
 Let $f\in C_0(\R^d)$ be such that $R_0f\in C_b^2(\R^d)$. Let $p\in(1,\infty)$ and assume $\eta$ satisfies \eqref{etaassumption}. Then
 \[\|\mathcal{B}R_0f\|_p\leq \frac14 \|f\|_p. \]
\end{proposition}
\begin{proof}
 Without loss of generality, we can assume $f\in L^p(\R^d)$. Otherwise the right-hand side of the assertion is infinite and the statement is trivially true.
 Using H\"older's inequality, we get 
 \begin{align*}
  \|\mathcal{B}R_0f\|_p & = \|(\mathcal{L}-\mathcal{L}_0)R_0f\|_p \\
  & = \Bigg\|\frac{1}{2}\sum_{j=1}^d \int_{\R\setminus\{0\}} (R_0f(\cdot+e_jh)-2R_0f(\cdot)+R_0f(\cdot-e_jh))\frac{c_{{\alpha_j}}}{|h|^{1+\alpha_j}}\, \d h\\
  & \qquad \times\left(|A_{jj}(\cdot)|^{\alpha_j} - |A_{jj}(x_0)|^{\alpha_j}\right) \Bigg\|_{p} \\
  & \leq \sum_{j=1}^d\Bigg\|\frac{1}{2}\int_{\R\setminus\{0\}} (R_0f(\cdot+e_jh)-2R_0f(\cdot)+R_0f(\cdot-e_jh))\frac{c_{{\alpha_j}}}{|h|^{1+\alpha_j}}\, \d h\\
  & \qquad \times\left(|A_{jj}(\cdot)|^{\alpha_j} - |A_{jj}(x_0)|^{\alpha_j}\right) \Bigg\|_{p} \\
  & \leq \frac12 \eta \sum_{j=1}^d \|\mathcal{M}_jR_0f\|_p.
 \end{align*}
Using \eqref{Lpj}, $\|\mathcal{M}_jR_0f\|_p\leq 2a(p^{\ast}-1)\|f\|_p$. Hence by the definition of $\eta$
\[  \|\mathcal{B}R_0f\|_p\leq \eta da(p^{\ast}-1)\|f\|_p\leq \frac{1}{4}\|f\|_p.\]
\end{proof}
\subsection{Boundedness of the resolvent}\label{bddsub}
The aim of this subsection is to prove that the resolvent operator for any weak solution to \eqref{SDEM} is bounded for any 
$f\in C_b^2(\R^d)$ by the $L^p$-norm of $f$.\\
More precisely, assume that $\mathds{P}$ is a solution to the martingale problem for $\mathcal{L}$ started at $x_0$, see Definition \ref{defmart} and $\mathds{E}$ the expectation with respect to $\mathds{P}$. Let 
\begin{equation}\label{slambda}
  S_\lambda f = \mathds{E}\left[\int_0^\infty e^{-\lambda t}f(X_t) \, \d t\right], \quad f\in C_b(\R^d).
\end{equation}
We want to show that under the assumption \eqref{etaassumption} there is a constant $c_1>0$ such that
\begin{equation}\label{bddresolvent}
|S_{\lambda} f| \leq c_1\|f\|_p.
\end{equation}
For each $n\in\N$ we first define the truncated process
\begin{equation}
 Y_t^n = \sum_{k=0}^\infty X_{k/2^n}\mathds{1}_{\{\frac{k}{2^n}\leq t < \frac{k+1}{2^n}\}\cap \{t\leq n\}}+X_n\mathds{1}_{\{t> n\}}
\end{equation}
and $U_t^n$ as the solution to the system of stochastic differential equations
\begin{equation}\label{SDEUn}
 dU_t^n = A(Y_{t-}^n) dZ_t, \ U_0^n = x_0, 
 \end{equation}
where $x_0\in\R^d$ is as in \eqref{SDEM}.
Since $Y_t^n$ is piecewise constant and constant after time $n$, for every $n\in\N$, there is a unique solution $U_t^n$ to \eqref{SDEUn}.
Let
\begin{equation}\label{vlambdan}
 V_\lambda^n f:=\mathds{E}\left[ \int_0^\infty e^{-\lambda t}f(U_t^n) \, \d t. \right], \quad f\in C_0(\R^d).
\end{equation}
First we show that the resolvent operator $V_{\lambda}^nf$ converges to $S_{\lambda}f$ for any $f\in C_b(\R^d)$ and any fixed $\lambda>0$.
This follows from the observation that $U_t^n(\omega)$ converges to $X_t(\omega)$ in probability for all $t\geq 0$. 
\begin{lemma}\label{resolconv}
Let $\lambda>0$. For all $f\in C_b(\R^d)$
\[\lim\limits_{n\to\infty} V_\lambda^n f = S_\lambda f.\]
\end{lemma}
\begin{proof}
Since $\lim\limits_{n\to\infty} Y_{t-}^n  = X_{t-}$
almost surely, by dominated convergence $U_t^n$ converges to $X_t$ in $L^p$ and therefore
$U_t^n$ converges to $X_t$ in probability with respect to the solution of martingale problem $\mathds{P}$.
Hence,
\[ \lim\limits_{n\to\infty} V_\lambda^n f = \lim\limits_{n\to\infty} \mathds{E}\left[\int_0^\infty e^{-\lambda t}f(U_t^n) \, \d t\right] = \mathds{E}\left[\int_0^\infty e^{-\lambda t}f(X_t) \, \d t\right]=S_\lambda f. \]
\end{proof}
The following Lemma gives an $L^p$-bound for $V_{\lambda}^n$, depending on $n$. 
Since we want to consider the limit of $V_{\lambda}^n$ for $n\to\infty$, the statement of this lemma is not sufficient for our purposes. 
Thus we have to improve the result to an uniform $L^p$-bound independent of $n$ afterwards. This is to be done in Theorem \ref{vlambda}.
\begin{lemma}\label{vlambdaabsch}
Let $p>\beta$, $n\in\N$ and $\lambda>0$. There is a constant $c_1>0$, depending on $n$, such that for all $f\in C_b(\R^d)$
\[ |V_{\lambda}^n f| \leq c_1\|f\|_p. \]
\end{lemma}
The result follows as in \cite[Lemma 5.1]{BCS}. Hence we omit it here.\\
Next, we prove for every $\lambda>0$ a uniform bound in $n$ for $\sup\limits_{\|f\|_{p}\leq 1} V_{\lambda}^nf.$
For this purpose, we need to define auxiliary functions.\\
Let $n\in\N$. For $s\geq 0$ and $\omega\in\Omega$ we define
\[ \widetilde{A^n_j}(s,\omega)=\sum_{k=0}^\infty A_{jj}(X_{\frac{k}{2^n}-}(\omega))\mathds{1}_{\{\frac{k}{2^n}\leq s < \frac{k+1}{2^n}\}\cap \{s\leq n\}}+A_{jj}(X_{n-}(\omega))\mathds{1}_{\{s\geq n\}} \]
and for $f\in C_b^2(\R^d)$
\[ \widetilde{\mathcal{L}_n}f(x,s,\omega):=\sum_{j=1}^d \int_{\R\setminus\{0\}} [f(x+e_j\widetilde{A^n_j}(s,\omega)h)-f(x)-h\mathds{1}_{\{|h|\leq 1\}}\partial_jf(x)\widetilde{A^n_{jj}}(s,\omega)]\frac{c_{\alpha_j}}{|h|^{1+\alpha_k}} \, \d h.\]
Moreover, let
\[\widetilde{\mathcal{B}_n}f(x,s,\omega)=\widetilde{\mathcal{L}_n}f(x,s,\omega)-\mathcal{L}_0f(x).\] 
Note that for each $j\in\{1,\dots,d\}$ by continuity it holds
\[ \lim\limits_{n\to\infty}\widetilde{A_{jj}^n}(s,\omega)=A_{jj}(X_{s-}(\omega)) \]
and by dominated convergence we have for all $f\in C_b^2(\R^d)$
\[ \lim\limits_{n\to\infty}\widetilde{\mathcal{L}_n}f(x,s,\omega)= \mathcal{L}f(X_{s-}(\omega)). \]
\begin{proposition}\label{Lptilde}
Let $f\in C_0(\R^d)$ such that $R_0f\in C_b^2(\R^d)$. Let $p\in(1,\infty)$ and assume that $\eta$ defined in \eqref{coeffeta} satisfies \eqref{etaassumption}. Then
\[ \|\widetilde{\mathcal{B}_n}R_0f\| \leq \frac14\|f\|_p. \]
\end{proposition}
\begin{proof}
Note that we can rewrite the operator $\widetilde{\mathcal{B}_n}$ by Lemma \ref{second-order} with weighted second order differences. \\
Since
\[  \sup\limits_{j\in\{1,\dots,d\}}\sup\limits_{(s,\omega)\in[0,\infty)\times\Omega}\left||\widetilde{A_{jj}^n}(s,\omega)|^{\alpha_j}-|A_{jj}(x_0)|^{\alpha_j}\right|\leq \eta, \]
we also have \eqref{etaassumption} if we take $\widetilde{A_{jj}^n}$ instead of $A_{jj}$. \\
The proof of Proposition \ref{Lptilde} is now similar to the proof of Proposition \ref{BR0}.
\end{proof}
We now improve the result of Lemma \ref{vlambdaabsch} by proving that there is an upper bound independent of $n$ such that the result holds. 
At first we prove the result for compactly supported functions whose resolvents are in $C_b^2(\R^d)$ and afterwards, in Corollary \ref{corrbdd}, we show by an elementary limit argument that the result
is true for functions in $C_b$ with resolvents in $C_b^2(\R^d)$.
\begin{theorem}\label{vlambda}
 Suppose $p>\beta$ and $\lambda>0$. There exists a constant
 $c_1>0$, such that for all $n\in\N$ and $g\in C_c(\R^d)$ with $R_0g\in C_b^2(\R^d)$
 \[ |V_\lambda ^n g|\leq c_1\|g\|_p. \]
\end{theorem}
\begin{proof}
The idea of the proof is to write $V_\lambda^nf$ in terms of $R_\lambda f$ and $\widetilde{B}R_\lambda f$ for a function $f\in C_b^2(\R^d)$ 
and use Proposition \ref{Lptilde} to get an upper bound independent of $n$. Similar to the proof of \cite[Theorem 5.3]{BCS} we can show that for $f\in C_b^2(\R^d)$
\begin{equation}\label{vln}
\begin{aligned}
V_\lambda^n f & = \frac{1}{\lambda} \mathds{E}[f(U_0^n)] + \frac{1}{\lambda}\mathds{E}\left[\int_0^\infty e^{-\lambda s}\widetilde{\mathcal{L}_n}f(U^n_s(\omega),s,\omega)\, \d s\right].
\end{aligned}
\end{equation}
Let $g\in C_c(\R^d)$ with $R_0 g\in C_b^2(\R^d)$. Then
\[ (\lambda-\mathcal{L}_0)R_{\lambda} g(x) = g(x) \iff \mathcal{L}_0R_{\lambda}g(x) = -g(x)+\lambda R_{\lambda}g(x). \]
Hence
\begin{equation}\label{resolid}
  \widetilde{\mathcal{L}_n}R_{\lambda}g(x,s,\omega) = \widetilde{\mathcal{B}_n}R_\lambda g(x,s,\omega)-g(x)+\lambda R_\lambda g(x).
\end{equation}
Let $f=R_{\lambda}g \in C_b^2(\R^d)$. Plugging \eqref{resolid} into \eqref{vln} for $f=R_{\lambda}g$ yields
\begin{align*}
V_\lambda^n R_\lambda g = \frac{1}{\lambda}\mathds{E}[R_\lambda g(U_0^n)] + \frac{1}{\lambda}\mathds{E}\left[\int_0^\infty e^{-\lambda s}\widetilde{\mathcal{B}_n}R_\lambda g(U^n_s(\omega),s,\omega)\, \d s\right] - \frac{1}{\lambda}V_\lambda^n g + V_\lambda^nR_\lambda g,
\end{align*}
which is equivalent to
\[ V_\lambda^n g = \mathds{E}[R_\lambda g(U_0^n)] + \mathds{E}\left[\int_0^\infty e^{-\lambda s}\widetilde{\mathcal{B}_n}R_\lambda g(U^n_s(\omega),s,\omega)\, \d s\right].\]
Let $h=g-\lambda R_{\lambda}g$. Then by Lemma \ref{resolnull}
\[ R_0h=R_0(g-\lambda R_{\lambda}g) = R_0g - \lambda R_0R_{\lambda} = R_{\lambda}g.\] 
Thus $R_0h\in C_b^2(\R^d)$.
Using the triangle inequality and Proposition \ref{resolventenabsch}, we get
\[ \|h\|_p \leq \|g\|_p+\|\lambda R_\lambda g\|_p \leq 2\|g\|_p. \]
Note
\begin{align*}
 \left| \mathds{E}\left[ \int_0^\infty e^{-\lambda s} \widetilde{\mathcal{B}_n}R_0h(U_s^n(\omega),s,\omega) \, \d s \right] \right| &\leq \mathds{E}\left[ \int_0^\infty e^{-\lambda s}|\widetilde{\mathcal{B}_n}R_0h(U_s^n(\omega),s,\omega)| \, \d s \right]\\
&=V_\lambda^n(|\widetilde{\mathcal{B}_n}R_0h(U_s^n(\omega),s,\omega)|). 
\end{align*}
We define 
\[\Theta_n:=\sup\limits_{\|g\|_p\leq 1}|V_\lambda^n g|. \]
By Lemma \ref{vlambdaabsch} there is a $c_2>0$, depending on $n$, but being independent of $g$, such that $|V_\lambda^n g|\leq c_2\|g\|_p$. 
Hence $\Theta_n\leq c_2<\infty$.\\
Now we need to find a constant, independent of $n$, such that the assertion holds. Note that we have shown $h\in C_0(\R^d)$ with $R_0h\in C_b^2(\R^d)$ which allows us to apply Proposition \ref{Lptilde} on $h$.
By Proposition \ref{resolventenabsch} and Proposition \ref{Lptilde} there exists a $c_3>0$, independent of $n$, such that
\begin{align*}
|V_\lambda^n g| &= \left|R_\lambda g(x_0) + \mathds{E}\left[\int_0^\infty e^{-\lambda s}\widetilde{\mathcal{B}_n}R_0 h(U_s^n(\omega),s,\omega) \, \d s\right]\right| \\
& \leq \left|R_\lambda g(x_0)\right| + \left|\mathds{E}\left[\int_0^\infty e^{-\lambda s}\widetilde{\mathcal{B}_n}R_0 h(U_s^n(\omega),s,\omega) \, \d s\right]\right| \\
& \leq c_3\|g\|_p + V_\lambda^n(|\widetilde{\mathcal{B}_n}R_0h(U_s^n(\omega),s,\omega)|) \\
& \leq c_3\|g\|_p + \Theta_n(\|\widetilde{\mathcal{B}_n}R_0h(U_s^n(\omega),s,\omega)\|_p) \\
& \leq c_3\|g\|_p + \Theta_n\left(\frac14 \|h\|_p\right) \leq \|g\|_p \left(c_3+\frac12 \Theta_n\right).
\end{align*}
Taking the supremum over all $g\in C_c(\R^d)$ with $R_0g\in C_b^2(\R^d)$ and $\|g\|_p\leq 1$, we get
\[ \Theta_n \leq c_3 + \frac12 \Theta_n \iff \Theta_n \leq 2c_3<\infty, \]
which proves the assertion for $g\in C_c(\R^d)$ with $R_0g\in C_b^2(\R^d)$.
\end{proof}
Note that we had to take $g\in C_c(\R^d)$ with $R_0g\in C_b^2(\R^d)$ in the proof of Theorem \ref{vlambda} so that the expressions in the proof are well-defined.
By a standard limit argument we conclude.
\begin{corollary}\label{corrbdd}
 Suppose $p>\beta$ and $\lambda>0$. There exists a constant
 $c_1>0$, such that for all $n\in\N$ and $f\in C_b(\R^d)$ with $R_0f\in C_b^2(\R^d)$
 \[ |V_\lambda ^n f|\leq c_1\|f\|_p. \]
\end{corollary}
\begin{proof}
By Theorem \ref{vlambda} we already know the result holds for compactly supported functions whose resolvents are in $C_b^2(\R^d)$. 
The assertion on $C_b(\R^d)$ instead of $C_c(\R^d)$ follows by dominated convergence. 
Let $f\in C_b(\R^d)$. Without loss of generality, we can assume $f\in L^p(\R^d)$. Otherwise
the right hand side of the assertion is infinite and the statement trivially holds true. 
Let $g_m\in C_c(\R^d)$ such that $g_m=f$ on $B_m$ and $\supp(g_m)\subset B_{m+1}$. Then $g_m\to f$ as $m\to\infty$.
Since $f\in L^p(\R^d)$, by dominated convergence $g_m$ also converges to $f$ in $L^p$. Moreover, since $g_m$ and $f$ are bounded, we have 
\[ \lim\limits_{m\to\infty} |V_{\lambda}^n g_m(x)| =  |V_{\lambda}^n \lim\limits_{m\to\infty} g_m(x)|, \]
which finishes the proof.
\end{proof}
Finally, we can prove the desired result.
\begin{corollary}\label{slambdabeschr}
 Suppose $p>\beta$. There exists a constant
 $c_1>0$, such that for all $f\in C_b(\R^d)$ with $R_0f\in C_b^2(\R^d)$
 \[ |S_\lambda f|\leq c_1\|f\|_p. \]
\end{corollary}
\begin{proof}
By Lemma \ref{resolconv} and Corollary \ref{corrbdd} we get
\[ |S_\lambda f|=\lim\limits_{n\to\infty} |V_\lambda^n f| \leq \lim\limits_{n\to\infty} c_1 \|f\|_p = c_1\|f\|_p, \]
where we have used the fact that $c_1$ is independent of $n$.
\end{proof}
The next proposition gives a representation of $S_{\lambda}f$. See \cite[Proposition 6.1]{BCS} for a proof.
\begin{proposition}\label{slambdaprop}
 Let $f\in C_b(\R^d)$ with $R_{\lambda}f\in C_b^2(\R^d)$ and $\lambda >0$. Then
 \[ S_\lambda f = R_\lambda f(x_0) + S_\lambda \mathcal{B}R_\lambda f. \]
\end{proposition}
\subsection{Auxiliary results}\label{auxsub}
Before we prove the main result, we prove some important technical results.
The following theorem states that in order to prove uniqueness of solutions to the martingale problem, 
it is sufficient to prove uniqueness for the corresponding resolvents. 
\begin{theorem}\label{laplacegleich}
 Let $\mathds{P}_1,\mathds{P}_2$ be two solutions to the martingale problem for $\mathcal{L}$ started at $x_0$.
 Suppose for all $x\in\R^d$, $\lambda>0$ and $f\in C_b^2(\R^d)$,
 \[ \mathds{E}_1\left[\int_0^{\infty} e^{-\lambda t} f(X_t) \, \d t \right]=\mathds{E}_2\left[\int_0^{\infty} e^{-\lambda t} f(X_t) \, \d t \right]. \]
 Then for each $x_0\in\R^d$ the solution to the martingale problem for $\mathcal{L}$ has a unique solution.
\end{theorem}
A proof of this theorem can be found e.g. in \cite[Theorem V.3.2]{BAD}. 
Although the author studies the martingale problem for the elliptic operator $\mathcal{A}$ in nondivergence form given on $C^2(\R^d)$ by
\[ \mathcal{A}f(x)=\frac12 \sum_{i,j=1}^d a_{ij}(x)\frac{\partial^2f(x)}{\partial x_i \partial x_j} + \sum_{i=1}^d b_i(x)\frac{\partial f(x)}{\partial x_i}, \]
where $a_{ij}$ and $b_i$ are bounded and measurable, the proof of Theorem \ref{laplacegleich} does not significantly change and does apply for a large class of operators. 
Hence, we don't give the proof and refer the reader to \cite[Theorem V.3.2]{BAD}. \\
Recall that assumption \eqref{etaassumption} states
\[\sup\limits_{j\in\{1,\dots,d\}} \||A_{jj}(\cdot)|^{\alpha_j}-|A_{jj}(x_0)|^{\alpha_j}\|_{L^{\infty}(\R^d)} \leq \frac{1}{4da(p^{\ast}-1)},\]
where $a=\max\{|A_{11}(x_0)|^{-\alpha_1},\dots,|A_{dd}(x_0)|^{-\alpha_d}\}$ and $p^{\ast}-1=\max\left\{(p-1),(p-1)^{-1}\right\}$.
By Proposition \ref{BR0} this assumption implies $\|\mathcal{B}R_0h\|_{p}\leq \frac14 \|h\|_p$ for $h\in C_c^2(\R^d)$.  \\
We first prove uniqueness of solutions to the martingale problem for $\mathcal{L}$ under the assumption \eqref{etaassumption}.
\begin{proposition}\label{uniqeta}
 Let $x_0\in\R^d$ and assume \eqref{etaassumption} holds for the coefficients of $\mathcal{L}$.
 Suppose $\mathds{P}_1$ and $\mathds{P}_2$ are two solutions to the martingale problem for $\mathcal{L}$ started at $x_0$. 
 Then $\mathds{P}_1=\mathds{P}_2$.
\end{proposition}
\begin{proof}
We follow the proof of \cite[Proposition 6.2]{BCS}.\\
Let $p>\beta$. Moreover let $S_\lambda^1$ and $S_\lambda^2$ be defined as above with respect to $\mathds{P}_1$ and $\mathds{P}_2$ respectively.
Set \[S_\lambda^{\Delta}g:=S_\lambda^1 g - S_\lambda^2 g,\] 
where $g\in C_b(\R^d)$ with $R_0g\in C_b^2(\R^d)$ and let
\[ \Theta=\sup\limits_{\|g\|_p\leq 1}|S_\lambda^{\Delta}g|. \]
By Corollary \ref{slambdabeschr}, we have $\Theta<\infty$. 
Let $f\in C_c(\R^d)$ with $R_0f\in C_b^2(\R^d)$ and define $h:=f-\lambda R_{\lambda}f$. As in the proof of Theorem \ref{vlambda} we conclude $h\in C_0(\R^d)$ and $R_0h=R_{\lambda}f\in C_b^2(\R^d)$. By Proposition \ref{BR0}, we have
\[ \|BR_0h\|_p \leq \frac14 \|h\|_p. \]
Furthermore, Proposition \ref{slambdaprop} and $\|h\|_p\leq 2\|f\|_p$ imply
\[ |S_\lambda^{\Delta}\mathcal{B}R_\lambda f| = |S_\lambda^{\Delta}\mathcal{B}R_0h| \leq \Theta\|\mathcal{B}R_0h\|_p \leq \frac14 \Theta \|h\|_p \leq \frac12 \Theta\|f\|_p. \]
As in the proof of Corollary \ref{corrbdd}, we can take $f\in C_b(\R^d)$ with $R_0f\in C_b^2(\R^d)$. 
Taking the supremum over $f\in C_b(\R^d)$ with $R_0f\in C_b^2(\R^d)$ and $\|f\|_p\leq 1$, we have 
$\Theta\leq \frac12 \Theta$ and since $\Theta$ is finite we have  $\Theta=0$ by Corollary \ref{slambdabeschr}.
Finally, taking $f\in C_b^2(\R^d)$, the result follows by Theorem \ref{laplacegleich}.
\end{proof}
The following result provides a maximal inequality.
\begin{lemma}\label{coroweak}
 Let $(\Omega,\mathcal{F},(\mathcal{F}_t)_{t\geq 0}, \mathds{P}, X, Z)$ be a weak solution to \eqref{SDEM}. There exists a constant $c>0$ depending
 only on the upper bounds of $|A_{ij}(x)|$ for $1\leq i,j\leq d$ and the dimension $d$, such that for every $\delta>0$ and $t\geq 0$ 
 \[ \mathds{P}\left( \sup\limits_{s\leq t} |X_s-X_0|>\delta  \right) \leq ct\sum_{j=1}^d \delta^{-\alpha_j}. \]
\end{lemma}
\begin{proof}
Using the maximal inequality from \cite[Theorem 5.1]{Bottch}, we immediately conclude
\[ \mathds{P}\left(\sup_{s\leq t} |X_s-X_0|\geq \delta \right) \leq c_1 t \sup_{|y-x|\leq r}\sup_{|\delta|\leq 1/r} \sum_{j=1}^d |\delta\cdot a_j(x)|^{\alpha_j}\leq c_2 t\sum_{j=1}^d \delta^{-\alpha_j}, \]
where $c_2$ depends on the upper bounds of $|A_{ij}(x)|$ for $1\leq i,j\leq d$ and the dimension $d$ only.
\end{proof}
Note that Lemma \ref{coroweak} is an immediate consequence of \cite[Lemma 4.1]{Schilling1}, where the author provides maximal inequalities for a large class of 
Feller processes. See also \cite[Section 3]{SchillingUemura}.

  Let $\Theta_t$ be the shift operator on $\mathds{D}([0,\infty))$ that is $f(s)\circ\Theta_t=f(s+t)$. Recall the definition of the first exit time:
 \[ \tau:=\tau_{B_r(x_0)} := \inf\{ t\geq 0 \colon |X_t-x_0|\geq r \}.\]
 We define 
 \[\mathds{P}_{\tau}(A)=\mathds{P}_i(A\circ \Theta_{\tau})\]
 and let $\mathds{E}_{\tau}$ be the expectation with respect to $\mathds{P}_{\tau}$.
 \begin{lemma}\label{rcplemma}
  Let $\mathds{P}$ be a solution to the martingale problem for $\mathcal{L}$ started at $x_0\in\R^d$ and 
  $\mathds{Q}(\cdot,\cdot)$ be a regular conditional probability for $\mathds{E}[\cdot\,|\mathcal{F}_{\tau}]$.
  Then $\mathds{Q}(\omega,\cdot)$ is $\mathds{P}$-almost surely a solution to the martingale problem for $\mathcal{L}$ started at $X_{\tau}(\omega)$.
 \end{lemma}
For a proof we refer the reader to \cite[Proposition VI.2.1]{BAD}.
\subsection{Proof of Theorem \ref{thm:uniqueness}}\label{proofsub}
\begin{proof}
The proof follows the idea of the proof of \cite[Theorem VI.3.6]{BAD}\\
Let $\mathds{P}_1$ and $\mathds{P}_2$ be two solutions to the martingale problem for $\mathcal{L}$ started at $x_0\in\R^d$. 
Recall that we consider the
canonical process $(X_t)_{t\geq 0}$ on the Skorohod space $\Omega=\mathds{D}([0,\infty);\R^d)$, i.e. $X_t(\omega)=\omega(t)$ and 
$(\mathcal{F}_t)$ is the minimal augmented filtration with respect to the process $(X_t)_t$. We denote the $\sigma$-field of the probability space by
$\mathcal{F}_{\infty}$.\\
For $N\in\N$ let 
\[ \rho_N:=\tau_{\overline{B_N}}:= \inf\{t\geq 0 : |X_t-x_0|>N\}. \]
Since c\`{a}dl\`{a}g functions are locally bounded the process $X_t$ does not explode in finite time. Further by the transience of $Z_t$, we have for $i=1,2$,
\begin{equation}\label{rho}
 \rho_N\to\infty \quad \mathds{P}_i\text{-a.s. } \quad \text{as } N\to\infty. 
\end{equation}
To prove $\mathds{P}_1=\mathds{P}_2$, we have to show that all finite dimensional distributions of $X_t$ under $\mathds{P}_1$ and $\mathds{P}_2$ are the same. 
By \eqref{rho} it is sufficient to show  that there is a $N_0\in\N$ such that for all $N\geq N_0$
\[ \mathds{P}_1 \bigr|_{\mathcal{F}_{\rho_N}}=\mathds{P}_2 \bigr|_{\mathcal{F}_{\rho_N}}.  \]
Choose $N_0=\lfloor |x_0| \rfloor + 1$ and let $N\geq N_0$ be arbitrary. Set 
\[ \|A\|_{\infty}:= \max\limits_{1\leq j\leq d} \sup\limits_{x\in\R^d} |A_{jj}(x)|. \]
Since $A(x)$ is non-degenerate at each point $x\in\R^d$, 
\[ \mu_1(A,N):= \inf\limits_{x\in B_N}\inf\limits_{u\in\R^d \colon \|u\|=1} |A(x)u|>0. \]
Let $\widetilde{\eta_0}:=\frac{\eta_0}{2}$, where $\eta_0$ is defined as in \eqref{etaassumption}. Since $A$ is continuous on $\R^d$, it is uniformly continuous on $\overline{B_{N+1}}$. 
Hence there is a $r\in(0,1)$ such that
\[ \sup\limits_{1\leq j \leq d} |A_{jj}(x)-A_{jj}(y)|<\frac{\widetilde{\eta_0}}{2} \quad \text{ for } x,y\in B_{N+1}, |x-y|<r. \]
Let $\widetilde{A}:\R^d\to\R^{d\times d}$ be diagonal such that $\widetilde{A}=A$ on $B_r$ and the functions $x\mapsto \widetilde{A_{jj}}(x)$ 
on the diagonal are continuous and bounded for all $j\in\{1,\dots,d\}$. Moreover let $\widetilde{A}$ be 
uniformly non-degenerate such that
\[ \mu_{1,1}(\widetilde{A}) = \inf\limits_{u\in\R^d: \\ |u|=1} \inf\limits_{x\in\R^d} |\widetilde{A}(x)u|>\frac{\mu_1(A,N)}{2} \]
and 
\[ \sup_{ j\in\{1,\dots,d\}} |\widetilde{A}_{jj}(\cdot)-\widetilde{A}_{jj}(x_0)|_{L^{\infty}(B_r)}<\widetilde{\eta_0}. \]
Let $\widetilde{\mathcal{L}}$ be defined as $\mathcal{L}$ with $A$ replaced by $\widetilde{A}$. 
By Proposition \ref{uniqeta}, there is a unique solution of the martingale problem for $\widetilde{\mathcal{L}}$ started at any $x_0\in\R^d$. 
We call this solution $\widetilde{\mathds{P}}$. \\
Let $\widetilde{\mathds{Q}}(\cdot,\cdot)$ be a regular conditional probability for $\widetilde{\mathds{E}}[\cdot\,|\mathcal{F}_{\tau}]$
By Lemma \ref{rcplemma} $\widetilde{\mathds{Q}}(\omega,\cdot)$ is $\widetilde{\mathds{P}}$-almost surely a solution to the martingale problem for $\widetilde{\mathcal{L}}$ started at $X_{\tau}(\omega)$.
For abbreviation we denote this measure by $\widetilde{\mathds{Q}}$ .\\
Define the measure on $(\mathcal{F}_{\infty}\circ\Theta_{\tau})\cap \mathcal{F}_{\tau}$ by
\[ \overline{\mathds{P}}(A\cap B\circ \Theta_{\tau}):= \int_{A}\widetilde{\mathds{Q}}(B) \, \d\mathds{P}_i, \quad A\in\mathcal{F}_{\tau}, \ B\in\mathcal{F}_{\infty}, \]
which represents the process behaving according to $\mathds{P}_i$ up to time $\tau$ and afterwards according to $\widetilde{\mathds{P}}$.\\
We now show that $\overline{\mathds{P}}_i$ solves the martingale problem for $\widetilde{\mathcal{L}}$ started at $x_0$.\\
Clearly $\overline{\mathds{P}}_i(X_0=x_0)=\mathds{P}_i(X_0=x_0)=1$. \\
Let $f\in C_b^2(\R^d)$. Then
\begin{align*}
 M_t &= f(X_{t\wedge\tau}) - f(X_0) - \int_0^{t\wedge\tau}\widetilde{\mathcal{L}}f(X_s)\, \d s\\
 & = f(X_{t\wedge\tau}) - f(X_0) - \int_0^{t\wedge\tau} \mathcal{L}f(X_s)\, \d s
\end{align*}
is $\mathcal{F}_{\tau}$ measurable for each $t\geq 0$ and by assumption a $\mathds{P}_i$-martingale. Therefore $(M_t)_{t\geq 0}$ a $\overline{\mathds{P}}_i$-martingale.
Further
\[ N_t=f(X_{t+\tau})-f(X_{\tau}) - \int_{\tau}^{t+\tau} \widetilde{\mathcal{L}}\,\d s \]
is a $\overline{\mathds{P}}_i$-martingale by Lemma \ref{rcplemma}.\\
Hence $\overline{\mathds{P}}_i$, $i=1,2$, is a solution to the martingale problem for $\widetilde{\mathcal{L}}$ started at $x_0$. By definition of $\widetilde{\mathcal{L}}$ the coefficients
satisfy the assumptions of Proposition \ref{uniqeta} and therefore $\overline{\mathds{P}}_1=\overline{\mathds{P}}_2$, 
which implies $\mathds{P}_1 \bigr|_{\mathcal{F}_{\tau}}=\mathds{P}_2 \bigr|_{\mathcal{F}_{\tau}}$.\\
We define the sequence of exit times $(\tau_{k})_{k\in\N}$ as follows
\[ \tau_1 : = \tau \quad \text{and} \quad \tau_{k+1} =\inf\{ t>\tau_k \colon |X_t-X_{\tau_k}|>r\}\wedge \rho_{N}. \]
Iterating the piecing-together method from before, we get $ \mathds{P}_1=\mathds{P}_2$ on $\mathcal{F}_{\tau_k}$ for all $k\in\N$.
By Lemma \ref{coroweak} it holds that $\tau_k\to \rho_N$ as $k\to\infty$ and hence we get $\mathds{P}_1=\mathds{P}_2$ on $\mathcal{F}_{\rho_N}$,
which finishes the proof.
\end{proof}
\subsection*{Acknowledgment} This work is part of the author's PhD thesis written under the supervision of Moritz Kassmann at Bielefeld University.
The author wishes to express his thanks to Franziska K\"uhn for helpful comments concerning the existence of solutions to the martingale problem.
The author gratefully acknowledges the many helpful suggestions of the anonymous referee.
\enlargethispage*{4 \baselineskip}
\bibliographystyle{alpha}
\bibliography{bib}

\end{document}